\documentclass[10pt]{amsart}
\usepackage{amsfonts,amsmath,amsthm,amssymb,color}
\usepackage[all]{xy}
\usepackage[pagebackref,colorlinks=true,linkcolor=blue,urlcolor=blue,citecolor=blue]{hyperref}
\usepackage{eucal,comment}
\numberwithin{equation}{section}
\providecommand{\eprint}[2][]{\href{http://arxiv.org/abs/#2}{arXiv:#2}}

\begin{document}

\newtheorem{theorem}{Theorem}[section]
\newtheorem{thm}[theorem]{Theorem}
\newtheorem{lemma}[theorem]{Lemma}
\newtheorem{proposition}[theorem]{Proposition}
\newtheorem{corollary}[theorem]{Corollary}

\theoremstyle{definition}
\newtheorem{definition}[theorem]{Definition}
\newtheorem{example}[theorem]{Example}

\theoremstyle{remark}
\newtheorem{remark}[theorem]{Remark}

\newenvironment{magarray}[1]
{\renewcommand\arraystretch{#1}}
{\renewcommand\arraystretch{1}}

\newcommand{\quot}[2]{
{\lower-.2ex \hbox{$#1$}}{\kern -0.2ex /}
{\kern -0.5ex \lower.6ex\hbox{$#2$}}}

\newcommand{\mapor}[1]{\smash{\mathop{\longrightarrow}\limits^{#1}}}
\newcommand{\mapin}[1]{\smash{\mathop{\hookrightarrow}\limits^{#1}}}
\newcommand{\mapver}[1]{\Big\downarrow
\rlap{$\vcenter{\hbox{$\scriptstyle#1$}}$}}
\newcommand{\liminv}{\smash{\mathop{\lim}\limits_{\leftarrow}\,}}

\newcommand{\Set}{\mathbf{Set}}
\newcommand{\Art}{\mathbf{Art}}
\newcommand{\solose}{\Rightarrow}

\newcommand{\specif}[2]{\left\{#1\,\left|\, #2\right. \,\right\}}

\renewcommand{\bar}{\overline}
\newcommand{\de}{\partial}
\newcommand{\debar}{{\overline{\partial}}}
\newcommand{\per}{\!\cdot\!}
\newcommand{\Oh}{\mathcal{O}}
\newcommand{\sA}{\mathcal{A}}
\newcommand{\sB}{\mathcal{B}}
\newcommand{\sC}{\mathcal{C}}
\newcommand{\sD}{\mathcal{D}}
\newcommand{\sE}{\mathcal{E}}
\newcommand{\sF}{\mathcal{F}}
\newcommand{\sG}{\mathcal{G}}
\newcommand{\sH}{\mathcal{H}}
\newcommand{\sI}{\mathcal{I}}
\newcommand{\sL}{\mathcal{L}}
\newcommand{\sM}{\mathcal{M}}
\newcommand{\sP}{\mathcal{P}}
\newcommand{\sU}{\mathcal{U}}
\newcommand{\sV}{\mathcal{V}}
\newcommand{\sX}{\mathcal{X}}
\newcommand{\sY}{\mathcal{Y}}
\newcommand{\sN}{\mathcal{N}}
\newcommand{\sT}{\mathcal{T}}
\newcommand{\Aut}{\operatorname{Aut}}
\newcommand{\Id}{\operatorname{Id}}
\newcommand{\Tr}{\operatorname{Tr}}
\newcommand{\Mor}{\operatorname{Mor}}
\newcommand{\Def}{\operatorname{Def}}
\newcommand{\Fitt}{\operatorname{Fitt}}
\newcommand{\Supp}{\operatorname{Supp}}
\newcommand{\Hom}{\operatorname{Hom}}

\newcommand{\Spec}{\operatorname{Spec}}
\newcommand{\Der}{\operatorname{Der}}
\newcommand{\coder}{\operatorname{Coder}}
 
\newcommand{\tot}{\operatorname{tot}}
\newcommand{\ten}{\bigotimes}
\newcommand{\mA}{\mathfrak{m}_{A}}

\newcommand{\N}{\mathbb{N}}
\newcommand{\R}{\mathbb{R}}
\newcommand{\Z}{\mathbb{Z}}

\newcommand{\K}{\mathbb{K}\,}
\newcommand\C{\mathbb{C}}
 
\newcommand\Tot{\operatorname{Tot}}


\newcommand{\rh}{\rightarrow}
\newcommand{\contr}{{\mspace{1mu}\lrcorner\mspace{1.5mu}}}

\newcommand{\bi}{\boldsymbol{i}}
\newcommand{\bl}{\boldsymbol{l}}

\newenvironment{acknowledgement}{\par\addvspace{17pt}\small\rm
\trivlist\item[\hskip\labelsep{\it Acknowledgement.}]}
{\endtrivlist\addvspace{6pt}}

\title[On the abstract BTT theorem]{On the abstract Bogomolov-Tian-Todorov Theorem}
\date{\today}

\author{Donatella Iacono}
\address{\newline  Universit\`a degli Studi di Bari,
\newline Dipartimento di Matematica,
\hfill\newline Via E. Orabona 4,
I-70125 Bari, Italy.}
\email{donatella.iacono@uniba.it}
\urladdr{\href{http://www.dm.uniba.it/~iacono/}{www.dm.uniba.it/~iacono/}}

\begin{abstract}
We describe an abstract version of the Theorem of  Bogomolov-Tian-Todorov, whose underlying idea is already contained in various papers by Bandiera, Fiorenza, Iacono, Manetti. More explicitly, we prove an algebraic criterion for a differential graded Lie algebras to be homotopy abelian.
Then, we  collect together  many examples and applications  in deformation theory and other settings.

\end{abstract}


\maketitle

\section{Introduction}

Let $X$ be a compact K\"{a}hler manifold. If  $X$ has trivial canonical bundle,
then, the so called Bogomolov-Tian-Todorov (BTT) Theorem states that the  deformations of  $X$ are  unobstructed.
In \cite{Bogomolov},  Bogolomov  proved it
 in the particular case of complex hamiltonian manifolds;
later, Tian \cite{Tian} and Todorov \cite{Todorov} proved  independently
the theorem for  compact K\"{a}hler manifolds
with trivial canonical bundle. 
More algebraic proofs of BTT theorem, based on $T^1$-lifting Theorem
and degeneration of the Hodge spectral sequence,  were  given in
\cite{zivran} for $\K=\mathbb{C}$  and in \cite{Kaw1,FM2}
for any $\K$ as above.

For  compact K\"{a}hler manifolds with torsion canonical bundle, the theorem follows from the more general fact that the derived infinitesimal deformations are unobstructed.
The guiding principle  is that in characteristic zero any deformation problem is controlled by a differential graded Lie algebra (DG-Lie algebra), with quasi-isomorphic DG-Lie algebras control the same deformation problem \cite{GoMil1,Kon}. More precisely, the deformation functor associated with the geometric problem is isomorphic to the deformation functor $\Def_L$ associated with a DG-Lie algebra $L$ via Maurer-Cartan equation up to gauge equivalence. 

 Therefore, it is worth to have an explicit description of a DG-Lie algebra associated with the problem. 
It turns out that the obstructions to the smoothness of the functor $\Def_L$ are contained in the cohomology vector space $H^2(L)$. However, if $L$ is an abelian DG-Lie algebra then, even if $H^2(L)$ is not zero,   the functor $\Def_L$ is smooth, i.e., it has no obtructions.
In particular, it is actually enough to prove that the DG-Lie algebra is homotopy abelian, 
 i.e., quasi-isomorphic to an abelian DG-Lie algebra, to assure that the associated
deformation functor is smooth.

For a compact complex manifold $X$, the infinitesiaml deformation are controlled by the Kodaira-Spencer DG-Lie algebras  $KS_X$ (Example \ref{example KS DGLA}). If $X$ is a compact K\"{a}hler manifold with trivial or torsion canonical bundle, then the Lie version of BTT theorem asserts that $KS_X$ is homotopy abelian (Section \ref{section def complex manifold}).
For $\K=\mathbb{C}$,  this was first proved in the seminal paper by  W.M. Goldman and  J.J. Millson \cite{GoMil2}, see also \cite{manRENDICONTi}.
For any algebraically closed field $\K$ of characteristic 0, this was proved in a completely algebraic way in \cite{algebraicBTT}, using the degeneration of the Hodge-to-de Rham spectral sequence and the notion of Cartan homotopy.

In \cite{algebraicBTT}, the proof  involves
$L_\infty$-algebras and $L_\infty$-morphisms and  it is based on a series of algebraic results that were also applied in other context \cite{semireg,donacoppie,Rugg,FM16}.
In this paper, we review some  of these ideas of independent interest  to establish the following  criterion,  that we call Abstract Bomolov-Tian-Todorov Theorem, for homotopy abelianity of a DG-Lie algebra (Theorem \ref{theorem abstract btt}).

 \begin{theorem}(Abstract BTT Theorem) Let $L$ and $ M$ be DG-Lie algebras over a field $\K$ of characteristic 0 and $H\subseteq M$ a DG-Lie subalgebra. Assume that there exists a linear map 
 \[ \bi\in \Hom^{-1}_{\K}(L,M),\qquad a\mapsto \bi_a,\]
 of degree $-1$ such that:
\begin{enumerate}
\item\label{item.ch} $[\bi_a,\bi_b]=0$ and $\bi_{[a,b]}=[\bi_a,d\bi_b]$ for every $a,b\in L$;\smallskip
\item\label{item.cc} $d \bi_a + \bi_{da} \in H$ for every $a\in L$;\smallskip
\item the inclusion $H \hookrightarrow M$  is injective in cohomology;\smallskip
\item the induced morphism of complexes $\bi\colon L \to (M/H)[-1]$ is injective in cohomology.\smallskip
\end{enumerate}
Then, the DG-Lie algebra $L$ is homotopy abelian.
\end{theorem}

According to \cite{Periods,algebraicBTT}, every linear map $\bi\colon L\to M$ of degree $-1$ that satisfies the previous condition \eqref{item.ch} is called Cartan homotopy (Section \ref{Section cartan homoto}). 
According to  \cite{FM16},  any pair $(\bi,H)$ satisfying the above conditions  \eqref{item.ch} and \eqref{item.cc} is called Cartan calculus.

The previous theorem is already  implicitly proved in  \cite{cone,algebraicBTT}, using 
$L_\infty$-algebras and $L_\infty$-morphisms.  Here, we show an alternative self contained proof based on the same ideas but involving  only DG-Lie  algebras.
The main motivation of this paper is to gather together various ideas in one place and  to provide an easier and more accessible proof (only DG-Lie algebras). Moreover, we collect many examples and applications of the Abstract BTT Theorem in deformation theory and in other derived settings (Section \ref{sez. example} and Section \ref{section further appli}). Among the others, we include the classical BTT Theorem for Calabi-Yau  manifolds of \cite{algebraicBTT}, the logarithmic version for log Calaby-Yau pairs $(X,D)$ of \cite{donacoppie}, the DG-Lie algebra associated with a Batalin-Vilkovisky algebra with the degeneration property \cite{FiMaformal,donacoppie}, and the DG-Lie algebra whose associated coderivation DG-Lie algebra has the splitting principle \cite{Rugg,bandieraFormal}.

The notion of homotopy abelianity   is closely related to the notion of formality. A DG-Lie algebra is  formal if it is quasi-isomorphic to its cohomology and so 
a DG-Lie algebra is homotopy abelian if and only if it is formal and $H^*(L)$ is abelian.  Then, homotopy abelianity condition is stronger  than formality and this could explain why it is easier to provide a criterion for it. Indeed, there are not so many analog of the previous theorem  that guarantees the formality of a DG-Lie algebra.  We refer to \cite{kaledin,lunts,FiMaformal,ManettiFormal,bandieraFormal}  and reference therein for formality criterion.

From the geometric point of view, homotopy abelianity assures the smoothness of the associated moduli problem while formality only implies that the singularity are not too bad (see \cite{GoMil1,ElenaFormal,ManettiFormal,LMDT} for more details).
If we are only interested in the smoothness of the problem, then instead of Hypothesis (4) it is enough to require that only $H^2(\bi)$ is injective (Remark \ref{remark dopo teorema diagramma semireg}).

The paper goes as follows: Section \ref{section back ground DGLA} is included for the non expert readers and it contains the relevant definitions and properties about DG-Lie algebras. Section \ref{Section cartan homoto} is devoted to the proof of the Abstract BTT Theorem. Some applications and examples are collected in Section \ref{sez. example}, while Section \ref{section further appli} contains some generalizations and further applications.

Throughout the paper, we work over an algebraically
closed field $\K$ of characteristic 0, if it is not differently specified.

\section{Background on DG-Lie algebras}\label{section back ground DGLA}

A \emph{differential graded Lie algebra} (DG-Lie algebra) is
the data of a triple $(L,d,[\ , \ ])$, where   $(L,d)$ is a differential graded vector space (DG-vector space) and    $ [\ , \ ] \colon L \times L \to L$ is a  bilinear map 
of degree 0 (called bracket), such that the following conditions are satisfied:
\begin{enumerate}

\item (graded  skewsymmetry)
$[x,y]=-(-1)^{ij} [y,x]\in L^{i+j}$,  for every $x\in L^i$ and $y\in L^j$;

\item (graded Jacobi identity) $ [x,[y,z]] = [[x,y],z] + (-1)^{ij} [y, [x,z]]$, for every $x\in L^i$, $y\in L^j$ and $z\in L$;

\item (graded Leibniz rule) $ d[x,y] =[ dx,y]+ (-1)^{i}[x, dy]$, for every $x\in L^i$ and $y\in L$.

\end{enumerate}

In particular, the Leibniz rule implies that the bracket of a DG-Lie algebra $L$ induces  a structure of graded Lie algebra on its cohomology $H^*(L)= \bigoplus_iH^i(L)$.
A DG-Lie algebra is called \emph{contractible} if $H^*(L)=0$.
A DG-Lie algebra is called \emph{abelian} if its bracket is trivial.
 
\begin{example}
If $L=\bigoplus_i L^i$ is a DG-Lie algebra, then $L^0$ is a Lie algebra in the
usual sense; vice-versa, every Lie algebra is a differential
graded Lie algebra concentrated in degree 0.

\end{example}
\begin{example}
Let $(V,d_V)$ be a differential graded vector space over $\K$ and $\operatorname{Hom}^i_\K(V,V)$
the space of the linear map $V\to V$ of degree $i$. Then,
$\operatorname{Hom}^*_\K(V,V)=\bigoplus_i\operatorname{Hom}^i_\K(V,V)$
is a DG-Lie algebra with bracket
\[
[f,g]=fg-(-1)^{\deg(f)\deg(g)}gf,
\]
and differential $d$ given by
\[
d(f)=[d_V,f]=d_Vf-(-1)^{\deg(f)}fd_V.
\]
For later use, we point out that by K\"{u}nneth formula, there exists
a natural isomorphism
\[ H^*(\Hom^*_\K(V,V))\xrightarrow{\;\simeq\;}\Hom^*_\K(H^*(V),H^*(V)).\]
\end{example}

\begin{example}\label{exe definizio M[t,dt]}
Let $M$ be a DG-Lie algebra and  $\K[t,dt]$  the differential graded algebra of  polynomial
differential forms over the affine line.  More precisely,
$\K[t,dt]=\K[t] \oplus \K[t]dt$, where $t$ has degree $0$ and $dt$
has degree 1. 
Then, $M[t,dt]=M \otimes \K[t,dt]$ is a DG-Lie algebra.
  As vector space $M[t,dt]$ is generated by elements
of the form $mp(t)+ nq(t)dt$, with $m,n \in M$ and $p(t),q(t)\in
\K[t]$. The differential and  the bracket on $M[t,dt]$ are defined
as follows:
$$
d(mp(t)+ nq(t)dt)=(dm)p(t) + (-1)^{\deg( m)} m p'(t)dt +
(dn)q(t)dt,
$$
$$
[mp(t),nq(t)]=[m,n]p(t)q(t), \ \ \
[mp(t),nq(t)dt]=[m,n]p(t)q(t)dt.
$$
Note that $[mdt,ndt]=0$, for every $m, n \in M$.
\end{example}

\begin{example}\label{example KS DGLA}

Let  $\Theta_X$ be the holomorphic tangent bundle of a complex
manifold $X$. The  \emph{Kodaira-Spencer} DG-Lie  algebra of $X$ is
$$
KS_X=(\bigoplus_i \Gamma(X,\sA_X^{0,i}(\Theta_X))=\bigoplus_i
A_X^{0,i}(\Theta_X) ,d , [ \ , \ ] ),
$$
where $KS_X^i = A_X^{0,i}(\Theta_X)$ is the vector space of the global sections
of the sheaf of germs of the differential $(0,i)$-forms with
coefficients in $\Theta_X$,  $d$ is the opposite of
Dolbeault's differential and the bracket is the extension of the usual bracket of vector fields.
Explicitly, if $z_1, \ldots,z_n$ are local holomorphic coordinates
on $X$, we have
$$
d\left(f d\overline{z}_I \frac{\de}{\de z_i}\right)=-\debar
(f)\wedge d\overline{z}_I \frac{\de}{\de z_i},
$$
$$
\left[f\frac{\de}{\de z_i}\, d\overline{z}_I,g\frac{\de}{\de z_j} \,
d\overline{z}_J\right]=\left(f\frac{\de g}{\de z_i}\frac{\de}{\de z_j}-
g\frac{\de f}{\de z_j}\frac{\de}{\de z_i}\right)\, d\overline{z}_I
\wedge d \overline{z}_J ,\qquad \forall  \ f,g \in \sA_X^{0,0}.
$$

\end{example}

\begin{example}\label{example DGLA pair}
Let $D$ be a submanifold of  a complex
manifold $X$ of codimension 1.  We denote by ${\Theta}_X(-\log D)$ the sheaf of germs of the tangent vectors to $X$ which are tangent to $D$ \cite[Section 3.4.4]{Sernesi}.
Denoting by $\mathcal{I}\subset \mathcal{O}_X$   the ideal sheaf of $D$ in $X$, then $\Theta_X(-\log  D)$ is the subsheaf of the derivations of the sheaf $\mathcal{O}_X$ preserving the ideal sheaf $\mathcal{I}$ of $D$. Moreover, we have the following short exact sequence
\[
0 \to {\Theta}_X(-\log D) \to {\Theta}_X \to N_{D / X} \to 0.
\]
Since we are in codimension 1, the sheaf  $\Theta_X (-\log D)$  is dual to the sheaf $\Omega^1_X(\log D)$ of logarithmic differentials, so it is in particular locally free, see for instance \cite[p. 72]{deligne},  \cite[Chapter 2]{librENSview} or \cite[Chapter 8]{voisin}.

Then, we can define the DG-Lie algebra of the pair $(X,D)$
$$
KS_{(X,D)}=(\bigoplus_i \Gamma(X,\sA_X^{0,i}(\Theta_X(-\log D))),  d, [\ , \ ]).$$
Note that $KS_{(X,D)}$ is a DG-Lie subalgebra of $KS_X$.
\end{example}

\subsection{Morphisms of DG-Lie algebras}

A \emph{morphism} of DG-Lie algebras is  a linear map $\varphi\colon L \to M$ 
that preserves degrees and commutes with brackets and
differentials. A \emph{quasi-isomorphism} of DG-Lie algebras is a morphism
that induces an isomorphism in cohomology.

Two  DG-Lie algebras $L$ and $M$ are said to be
\emph{quasi-isomorphic} if they are equivalent under the
equivalence relation generated by 
 quasi-isomorphisms.
 
\begin{example}
  
Let $M$ be a DG-Lie algebra and  $M[t,dt]$ the DG-Lie algebra introduced in Example \ref{exe definizio M[t,dt]}.
Then, for every $a \in \K$,  we can define the evaluation morphism
$$
e_a:M[t,dt] \to M,
$$
$$
e_a(\sum m_it^i +n_it^i dt)=\sum m_i a^i.
$$
Note that, every $e_a$ is a morphism of DG-Lie algebras which is a left
inverse of the inclusion $i:M \to M[t,dt]$, i.e., $e_a \circ i = \Id_M$. 
In particular,   $e_a$  is also a surjective quasi-isomorphism. We often use the short notation $m(a)=e_a(m(t,dt))$, for every $m(t,dt) \in M[t,dt]$.

\end{example}

For every two morphisms of DG-Lie algebras $f:L \to N$ and $g: M \to N$, we can consider the pull-back:
 \begin{center}
$\xymatrix{L\times_N M \ar[r]^{f'}  \ar[d]^{g'}     &M  \ar[d]^{g} \\
          L \ar[r]_f&  N.  \\ }$
\end{center}
Note that if $g$ (or $f$) is a surjective quasi-isomorphism then  $g'$ (or $f'$) is also a surjective quasi-isomorphism.

\begin{lemma} \label{lem. factorisation}(Factorisation Lemma)
Let $f:L \to M$ be a morphism of  DG-Lie algebras, then there exists a DG-Lie algebra $P$ and a factorisation
\begin{center}
$\xymatrix{L \ar[rr]^f\ar[dr]_{i} &  &M \\
          & P,\ar[ru]_g &    \\ }$
\end{center}
such that $g: P\to M$ is a surjective morphism and $i:L\to P$ is an injective quasi-isomorphism  (which is a right inverse of a surjective quasi-isomorphism).

 \end{lemma}
 
\begin{proof}
An explicit factorisation can be defined as follow.
Let $P_f$ be the   DG-Lie algebra
\[ P_f= \{(x,m(t,dt))\in L\times M[t,dt]\mid m(1)=f(x)\};\]
note that $P_f$ is given by the   pull back diagram
\begin{center}
$\xymatrix{P_f \ar[r] \ar[d]_p   &M[t,dt] \ar[d]^{e_1} \\
          L \ar[r]_f&  M,  \\ }$
\end{center}
where $p$ is the projection on the first factor.
In particular, since $e_1$ is a surjective quasi-isomorphism, $p$ is also a surjective quasi-isomorphism.
Next, define 
\[
i:L\to P_f  \qquad i(x)=(x,f(x)), \qquad \forall \ x \in L;
\]
and
\[
g\colon P_f \to M \qquad g(x, m(t,dt))= m(0), \qquad \forall  \  (x, m(t,dt)) \in P_f.
\]
The morphism $g$ is surjective: for any $ m \in M$, there exists $(0,(1-t)m) \in P_f$  such that $g( 0,(1-t)m)=e_0((1-t)m)=m$.
As regard the morphism $i$, it is a right inverse of $p$, i.e., $p i=\Id_L$ and so it is an injective quasi-isomorphism.
Finally, we have  
\[(g \circ i) (x)= g(x,f(x))=f(x),  \qquad \forall \ x \in L.
\]
\end{proof}

\begin{corollary}\label{coro. unico tetto per quasi iso}
Let $L$ and $M$ be  DG-Lie algebras. Then, $L$ and $M$  are quasi-isomorphic if and only if there exist  a DG-Lie algebra $P$ and two surjective quasi-isomorphisms $p:P \to L$ and $q: P \to M$.

\end{corollary}
\begin{proof}
The DG-Lie algebras $L$ and $M$ are quasi-isomorphic if and only if there exists a sequence of quasi-isomorphisms 
\[\xymatrix{&K_1\ar[dr]\ar[dl]&&K_2\ar[dr]\ar[dl]& &    &K_n\ar[dr]\ar[dl]&  \\
L&&H_1&&H_2\!\! & \cdots \quad H_{n-1} & &M.}\]
By Factorisation Lemma \ref{lem. factorisation}, we can assume that all the quasi-isomoprhisms in the sequence are surjective.
Indeed, consider the following diagram
\[\xymatrix{&K \ar[dr]^t\ar[dl]_f&&   \\
L&&N}\]
and apply  Factorisation Lemma \ref{lem. factorisation} to the morphism $(f,t)\colon K\to L\times N$ in order to obtain a diagram of quasi-isomorphisms
\[\xymatrix{
&K\ar[dr]^t\ar[dl]_f\ar[d]& \\
L&P\ar[l]\ar[r]&N,  }\]
where the two horizontal arrows are surjective.
Finally, any sequence of surjective quasi-isomorphisms can be replaced by two surjective quasi-isomorphisms using fibre product   and the fact that surjective quasi-isomorphisms are stable under pull backs
\[\xymatrix{&&K_1\times_{H_1}K_2 \ar[dr]\ar[dl]&&\\
&K_1\ar[dr]\ar[dl]&&K_2\ar[dr]\ar[dl]&\\
L&&H_1&&M.}\]

\end{proof}

\begin{definition}
A  DG-Lie algebra $L$ is called  \emph{formal} if it  quasi-isomorphic to $H^*(L)$ (intended as a DG-Lie algebra with trivial differential).

A  DG-Lie algebra $L$ is called  \emph{homotopy abelian} if it is quasi-isomorphic to an abelian DG-Lie algebra.
\end{definition}

\begin{example}
Any DG-vector space is formal (and  abelian as  DG-Lie algebra).  
Let $(V,d_V)$ be a  DG-vector space, then the  DG-Lie algebra $\Hom^*_\K(V,V)$ is formal.
\end{example}

%
%
%

\begin{lemma} \label{lemma transfer}(Transfer Lemma)
Let $f:L\to M$ be a morphism of DG-Lie algebras and denote by $H^*(f): H^*(L) \to H^*(M)$ the induced morphism in cohomology.
 \begin{enumerate} 
\item  If $M$ is homotopy abelian and  $H^*(f)$ is injective, then $L$ is also homotopy abelian. \\
\item  If $L$ is homotopy abelian and  $H^*(f)$ is surjective, then $M$ is also homotopy abelian. \\
\end{enumerate}
\end{lemma}

\begin{proof}
There are various proofs of this fact, see for instance \cite[Proposition 4.11]{KKP}, \cite[Lemma 1.10]{algebraicBTT}.
We follow the proof given in \cite[Lemma 2.7]{FM16}.
As regard (1),  since $M$ is  homotopy abelian, Corollary \ref{coro. unico tetto per quasi iso} implies the existence of an abelian DG-Lie algebra $A$ and two surjective quasi-isomorphisms
\[\xymatrix{&K \ar[dr]^a\ar[dl]_m &&   \\
M&&A.}\]
Applying the pull back we obtain the diagram
\[\xymatrix{&L\times_M K\ar[dl]_{m'} \ar[r]^{\ f'}&K \ar[dr]^a \ar[dl]_{m} &&   \\
L\ar[r]^f&M &&A,}\]
where $m, m'$ and  $a$ are surjective quasi-isomorphisms and the morphisms  $f$ and $f'$ induce injective morphisms in cohomology.

To conclude the proof it is enough to consider a graded vector space $E$ with a projection $e: A\to E$, such that $eaf'$ is a quasi-isomorphism:
\[\xymatrix{&L\times_M K\ar[dl]_{m'} \ar[r]^{\ f'}&K \ar[dr]^a  && &  \\
L & &&A \ar[r]^e&E.}\]

As regard (2), since $L$ is  homotopy abelian, Corollary \ref{coro. unico tetto per quasi iso} implies the existence of an abelian DG-Lie algebra $A$ with trivial differential and two surjective quasi-isomorphisms
\[\xymatrix{&K \ar[dr]^l\ar[dl]_a &&   \\
A&&L.}\]
Then, we can choose a graded Lie algebra $ H$ together with a morphism $ h: H\to A$ such that the composition
\[
H \stackrel{h}{\longrightarrow} A  \stackrel{H^*(a)^{-1}}{\longrightarrow} H^*(K) \stackrel{H^*(l)}{\longrightarrow} H^*(L)  \stackrel{H^*(f)}{\longrightarrow} H^*(M),  \]
is an isomorphism.
Finally, taking the   fibre product of $h$ and $a$, we obtain a commutative diagram
\[\xymatrix{ &H\times_A \times K \ar[dl]_{a'} \ar[r]^{\ \ h'} &K \ar[dr]^l\ar[dl]_a && &   \\
H\ar[r]^h &A&&L\ar[r]^f &M,}\]
where $a, a'$  and $f l h'$ are   quasi-isomorphisms.

\end{proof}

\begin{remark}\label{remark H^2 inje basta}
From the point of view of deformation theory, we could be only  interested in analysing the obstruction problem. We already noticed that the obstructions of the deformation functor $\Def_L$ associated with a DG-Lie algebra $L$ are contained in the vector space $H^2(L)$. Moreover, any morphism of DG-Lie algebras $f:L\to M$ induces a morphism $f: \Def_L \to \Def_M$ of the associated  deformation functors,  that behaves well with respect to the obstructions. If $M$ is homotopy abelian, then $\Def_M$ is smooth. Therefore, it is enough   that the morphism $H^2(f): H^2(L) \to H^2(M)$ is injective for the smoothness of the functor $\Def_L $.

\end{remark}

\begin{definition}
Let $f:L\to M$ be a morphism of DG-Lie algebras, the homotopy fibre  of $f$ is defined as the DG-Lie algebra
\[
TW(f) = \{(x, m(t,dt)) \in L \times M[t,dt] \ \mid \  m(0)=0, \, m(1)=f(x) \}.
\]
\end{definition}
Note that the projection $TW(f)\to L$ is a morphism of DG-Lie algebras.

\begin{remark}\label{oss TW funtoriale}
Let $f:L\to M$ be a morphism of DG-Lie algebras and  $L\stackrel{i}{\to} P_f\stackrel{g}{\to} M$   the explicit  factorisation, given in the proof of Factorisation Lemma 
\ref{lem. factorisation}. Then, $TW(f) =\ker g$.  
It can be proved that for any other factorisation  $L\stackrel{i'}{\to} P'\stackrel{g'}{\to} M$, the kernel $\ker g'$ is quasi-isomorphic to $TW(f)$ \cite[Section 6.1]{LMDT}. Moreover, every  commutative diagram of morphisms of DG-Lie algebras:
\[\xymatrix{ L \ar[d]  \ar[r]^f & M\ar[d]     \\
 L '\ar[r]^{f'}& M',}\]
induces a morphism of the  homotopy fibres $TW(f) \to TW(f')$. 
\end{remark}

\begin{remark} \label{rem.quasiisoTWcono} 

If $f\colon L\to M$ is an injective morphism of DG-Lie algebras, 
then its cokernel $M/f(L)$ is a DG vector space and the map
\[ 
TW(f)\to M/f(L)[-1]
\]
\[
(x,p(t)m_0+q(t)dt m_1)
\mapsto 
\left(\int_0^1q(t)dt\right) m_1 \pmod{f(L)},
\]
is a surjective quasi-isomorphism.
\end{remark}

\begin{lemma}\label{lem se injec hom fiber  quasi abelian}(Homotopy fibre)
Let $f:L\to M$ be a morphism of DG-Lie algebras. 
\begin{enumerate}
\item If the induced morphism $H^*(f): H^*(L) \to H^*(M)$ is injective,
then the $TW(f)$  is   homotopy abelian.
\item If the induced morphism $H^1(f): H^1(L) \to H^1(M)$ is injective,
then  $\Def_{TW(f)}$  is  unobstructed.
\end{enumerate}

\end{lemma}

\begin{proof}
 \cite[Proposition 3.4]{algebraicBTT} , \cite[Lemma 2.1]{semireg} or \cite[Corollary 2.8]{FM16}.
As regard (1), consider the DG vector space $M[-1]$ as an abelian DG-Lie algebra and the morphism of DG-Lie algebras
\[
\rho: M[-1] \to TW(f) \qquad \forall \ m \in M \qquad \rho(m)=   (0 , dt m).
\]
According to (2) of  Lemma \ref{lemma transfer}, it is enough to show that $\rho$ induces a surjective map in cohomology but this follows from the exact sequence
 \[\cdots \to H^{i-1}(M)  \to H^{i}(TW(f))\to H^{i}(L) \xrightarrow{H^i(f)} H^{i}(M) \to  \cdots,
\]
since  the morphisms $H^i(f)$ are all injective by hypothesis.

As regard (2), the morphism  of DG-Lie algebras  $\rho: M[-1] \to TW(f)$ induces a morphism of deformation functors $\rho:\Def_{M[-1]} \to \Def_{TW(f)}$, with $\Def_{M[-1]}$ a smooth functor. Then, by the Standard Smoothness Criterion \cite{FM2}, \cite[Theorem 4.11]{ManettiSeattle}, since $\rho$ is injective on obstructions, if $H^1(\rho)$ is surjective then $\Def_{TW(f)}$ is smooth. By the above exact sequence, if $H^1(f): H^1(L) \to H^1(M)$ injective then $H^1(\rho)$ is surjective.

\end{proof}

%
%

\begin{example}

 Let $W$ be a differential graded vector space  $U \subset W$
be a DG subspace. If  the induced  morphism in cohomology
 $H^*(U)\to H^*(W)$ is  injective, then the inclusion  of DG-Lie algebras
\[
f \colon \{f\in \Hom^*_{\K}(W,W) \mid f(U) \subset U\} \to \Hom^*_{\K}(W,W)
\]
satisfies  the hypothesis of Lemma \ref{lem se injec hom fiber  quasi abelian} and so the DG-Lie algebra $TW(f)$ is   homotopy abelian
\cite[Example 3.5]{algebraicBTT} and \cite[Proposition 5.10]{FM16}.
 
The deformation functor associated with the DG-Lie algebra $TW(f)$  has a natural interpretation as the local structure of the derived Grassmannian of $W$ at the point $U$. Therefore, the derived Grassmannian of $W$ is smooth at the points corresponding to subspaces $U$ such that  $H^*(U) \to H^*(W)$ is injective \cite{FM16}.

\end{example}

\section{Cartan homotopies and Main Theorem}\label{Section cartan homoto}

Let $L$ and $M$ be two DG-Lie algebras. A \emph{Cartan homotopy} is a linear map of degree $-1$
\[ \bi \colon L \to M  \]
such that,  for every $a,b\in L$, we have:
\[ \bi_{[a,b]}=[\bi_a,d_M\bi_b] \qquad  \text{and }   \qquad [\bi_a,\bi_{b}] =0.\]
 For every Cartan homotopy $\bi$, it is defined the Lie derivative  map
\[ \bl \colon L\to M,\qquad
\bl_a=d_M\bi_a+\bi_{d_L a}.
\]
It follows from the definition of $\bi$ that  $\bl$ is a morphism of DG-Lie algebras and we can write the conditions   of being a Cartan homotopy as
\[\bi_{[a,b]}=[\bi_a,\bl_b] \qquad \text{and } \qquad [\bi_a,\bi_{b}]=0.\]
Note that, as a morphism of complexes, $\bl$ is homotopic to 0 (with homotopy $\bi$).

\begin{example}\label{exam.cartan su ogni aperto}

Let $X$ be a smooth variety. Denote by $\Theta_X$ the  tangent
sheaf and by $(\Omega^{\ast}_X,d)$ the algebraic de Rham complex.
Then, for every open subset $U \subset X$,  the contraction of a vector   with a differential form
\[ 
\Theta_X(U) \otimes \Omega^k_X(U) \xrightarrow{\quad\contr\quad}
 \Omega^{k-1}_X(U)
\]
induces a linear  map of degree $-1$
\[
\bi \colon  \Theta_X(U) \to  \Hom^*(\Omega^{*}_X(U),
\Omega^{*}_X(U)), \qquad  \bi_{\xi} (\omega) = \xi \contr\omega
\]
that  is a Cartan homotopy. Indeed, the above conditions coincide 
with the classical Cartan's homotopy formulas. 
\end{example}

\begin{example}\label{exam.cartan relativo su ogni aperto}

Let $D$ be a smooth subvariety of codimension 1  of a smooth variety $X$.
Let $(\Omega^{\ast}_X(\log D),d)$  be the logarithmic differential complex and $\Theta_X(-\log  D)$   the logarithmic  tangent sheaf. It is easy to prove explicitly that for every open subset $U\subset X$, we have
\[(\, \Theta_X(-\log  D)(U)\ \contr\ \Omega^k_X(\log D)(U)\,) \subset
\Omega^{k-1}_X(\log D)(U).\]
Then, as above, the induced linear map of degree $-1$
\[
\bi\colon \Theta_X(-\log  D)(U)\to \Hom^*(\Omega^{*}_X(\log  D)(U),
\Omega^{*}_X(\log  D)(U)),\qquad \bi_{\xi}(\omega)=\xi\contr\omega
\]
is a Cartan homotopy.
\end{example}

We are now ready to prove the main theorem.

 \begin{theorem}\label{theorem abstract btt}(Abstract BTT Theorem) Let $L$ and $ M$ be DG-Lie algebras over a field $\K$ of characteristic 0 and $H\subseteq M$ a DG-Lie subalgebra. Assume that there exists a linear map 
 \[ \bi\in \Hom^{-1}_{\K}(L,M),\qquad a\mapsto \bi_a,\]
 of degree $-1$ such that:
\begin{enumerate}
\item $[\bi_a,\bi_b]=0$ and $\bi_{[a,b]}=[\bi_a,d\bi_b]$ for every $a,b\in L$;\smallskip
\item $d \bi_a + \bi_{da} \in H$ for every $a\in L$;\smallskip
\item the inclusion $H \hookrightarrow M$  is injective in cohomology;\smallskip
\item the induced morphism of complexes $\bi\colon L \to (M/H)[-1]$ is injective in cohomology.\smallskip
\end{enumerate}
Then, the DG-Lie algebra $L$ is homotopy abelian.
\end{theorem}

\begin{proof}
Let $s$ be a formal variable of degree $-1$ and consider the commutative DG-algebra $\K[s]$; note that $s^2=0$ and the  differential $d$  on $\K[s]$ is defined as
\[
d: \K s \to \K, \qquad d(s)=1.
\]
In particular, $\K[s]$ is a contractible DG  algebra  and the inclusion $\K \hookrightarrow \K[s]$ is a morphism of DG  algebras. 

Next, we consider the  DG-Lie algebra $ \K[s]\otimes L$.  For all $s\otimes a \in \K[s]\otimes L$, we have   $deg(s\otimes a) = deg(a) -1$ and
\[
d(s\otimes a)= 1\otimes a- s\otimes da.
\]
Then, we can define  a morphism of DG-Lie algebras by
\[
\varphi: \K[s]\otimes L \to M, \qquad \varphi(s\otimes a)= \bi_a;
\]
in particular,   $\varphi(1\otimes a)= \varphi ( d(s\otimes a )+s\otimes da)= d(  \varphi (s\otimes a))+\bi_{ da}= d \bi_a + \bi_{ da}=\bl_a$, for any $a \in L$,and it  is contained in $H$ by Hypothesis (2).

Thus, we can construct a commutative diagram of morphisms of DG-Lie algebras
\[\xymatrix{ L \ar[d]_\alpha \ar[r]^\psi & H\ar@{^{(}->}^\chi[d]     \\
\K[s]\otimes L \ar[r]^\varphi& M,}\]
where $\alpha(a) = 1\otimes a$ and $\psi(a)= \bl_a$,  for any $a \in L$.

According to Remark \ref{oss TW funtoriale}, this diagram induces a morphisms of DG-Lie algebras
\[
\phi: TW(\alpha) \to TW(\chi).
\]
Hypothesis (3) and Lemma \ref{lem se injec hom fiber  quasi abelian} applied to the inclusion  $\chi: H \to M$ imply  that the DG-Lie algebra  $TW(\chi)$ is homotopy abelian. Moreover, according to Remark \ref{rem.quasiisoTWcono},  $TW(\chi)$  is quasi-isomorphic as a DG vector spaces to $(M/ H)[-1]$.
Since $\K[s]\otimes L $ is contractible, then $TW(\alpha) \to L$ is a quasi-isomorphism. Finally, Hypothesis (4) implies that $\phi$ is injective in cohomology and so Lemma \ref {lem se injec hom fiber  quasi abelian}  implies that $TW(\alpha)$ is homotopy abelian. It follows that    $L$ is also homotopy abelian.

%
%
%
%

\end{proof}

\begin{remark}\label{remark dopo teorema diagramma semireg}
In the above notation and using the first three hypothesis, we have constructed a diagram:
\[\xymatrix{ TW(\alpha) \ar[d] \ar[r]^\phi  & TW(\chi)     \\
 L,}\]
where the vertical map is a quasi-isomorphism. Then, Hypothesis (4) implies that the horizontal map is injective in cohomology.
If we are only interested in the analysis of the obstruction of $\Def_L$, 
then Hypothesis (3) can be relaxed to the condition that only $H^1(\chi)$ is injective and
Hypothesis (4) can be relaxed to the condition that only $H^2(\bi)$ is injective.
Indeed,  $\Def_{TW(\chi)}$ is unobstructed, and so for the vanishing of the obstructions of $\Def_L$ it is enough that $H^2(\phi)$ is injective
(see Section \ref{section further appli} for further generalizations).
\end{remark}

\section{Examples and Applications}\label{sez. example}

In this section, we collect some applications of the main Theorem \ref{theorem abstract btt}.

\subsection{Deformations of compact  manifolds}\label{section def complex manifold}

Let $X$ be a holomorphic compact manifold and denote by $\Theta_X$ its holomorphic tangent bundle. We introduced the Kodaira Spencer DG-Lie algebra $KS_X$ in Example \ref{example KS DGLA}.
Let $(A_X^*,d)=( \bigoplus_{p,q} \Gamma(X,\sA_X^{p,q}), d=\de+\debar) $ be the De Rham complex of $X$, where $\sA_X^{p,q}$ denotes the sheaf of $(p,q)$-differential forms on X and   consider the DG-Lie Algebra  $M= \Hom_\C^*( A_X^*,A_X^*)$.
 
\begin{corollary}\label{cor def X non ostruite}
Let $X$ be a compact manifold with torsion canonical bundle such that the Hodge-de Rham spectral sequence degenerates at $E_1$-level. Then, $KS_X$ is homotopy abelian.
\end{corollary}

\begin{proof}
\cite[Corollary 2]{zivran}, \cite[Corollary B]{manetti adv}, \cite[Corally 6.5]{algebraicBTT}.
Let us first consider the case in which the canonical bundle is trivial.
The contraction of vector fields and differential forms together with the cup product defines a Cartan homotopy \cite[Section 6]{ManettiSeattle}
\[\bi\colon KS_X \to M=\Hom^{*}(A_X^*, A_X^*),\quad
\bi_\eta( \omega ) = \eta \contr  \omega, \qquad \forall \ \eta \in KS_X,
\ \  \forall \ \omega\in
A_X.\]

Moreover, for any $ \eta \in KS_X$, $\bl_\eta \in H=\{  \varphi \in M \ \mid \varphi( A_X^{n,*}  )\subset A_X^{n,*}\} $, where $n$ is the dimension of $X$.
By the degeneration property the inclusion $H \hookrightarrow M$  is injective in cohomology.
The  canonical bundle is trivial and so   the cup product with a non trivial section of it, gives the isomorphisms $H^i(\Theta_X)\cong H^i(\Omega^{n-1}_X)$, where $\Omega^{n-1}_X$ denotes the sheaf of holomorphic differential $n-1$ forms.
Then, $H^*(KS_X) \to \Hom^{*}(H^0(\Omega^{n}_X), H^*(\Omega^{n-1}_X) )$ is injective  and this implies that $KS_X \to M/H[-1]$ is also injective in cohomology.
Therefore, the hypothesis of   Theorem \ref{theorem abstract btt} are satisfied.

In the case of torsion canonical bundle,  there exists $m>0$ such that $K_X^{ m}= \Oh_X$, where $K_X$ denote the canonical bundle.
Next, consider the  unramified  $m$-cyclic cover, defined by $K_X$, i.e., 
$\pi: Y= \Spec ( \bigoplus_{i=0}^{m-1} L^{-i}) \to X $ (see \cite{pardini} for full details on abelian covers).
 Then, $\pi :Y \to X$  is a finite flat  map of degree $m$ and  $Y$ is a compact manifold with trivial canonical bundle, since $K_Y \cong \pi^* K_X\cong \Oh_Y$. Therefore, $KS_Y$ is homotopy abelian.
Finally, it is enough to consider the  morphism of DG-Lie algebras $KS_X \to KS_Y$ induced by pull back.
This morphism  is injective in cohomology  and so (1) of Lemma \ref{lemma transfer} implies that $KS_X$ is also homotopy abelian.

\end{proof}

\begin{remark}
 The degeneration hypothesis is satisfied if the $\de\debar$-Lemma holds, for instance for  K\"{a}hler  manifolds.

\end{remark}
 
 \begin{remark}
It is well known that the  Kodaira Spencer DG-Lie algebra $KS_X$ controls the infinitesimal deformations of $X$. Then,  
the infinitesimal deformations of a compact K\"{a}hler manifold with torsion canonical bundle $X$ are unobstructed. This is the classical Bogomolov-Tian-Todorov Theorem.\end{remark}

\begin{remark} 

Let $X$  be a smooth projective variety   over an algebraically closed field $\K$ of characteristic 0. Then, the analogous of Corollary \ref{cor def X non ostruite} holds.

 In this case, we can replace the Kodaira-Spencer DG-Lie algebra with 
the DG-Lie algebra  $\Tot(\Theta_X(\mathcal{U}))$,  obtained  applying the Thom-Withney totalization to the semicosimplicial DG-Lie algebras
$\Theta_X(\mathcal{U})$, for any affine open cover $\mathcal{U}$ of $X$ \cite[Theorem 5.3]{algebraicBTT}.
Also in this case,  if the canonical bundle of $X$ is torsion,
then   $\Tot(\Theta_X(\mathcal{U}))$  is homotopy-abelian and so   $X$ has unobstructed deformations \cite[Theorem 6.2 and Corollary 6.5]{algebraicBTT}.
\end{remark}

\subsection{Deformations of pairs (divisor, manifold)}

Let $D$ be  a smooth divisor in a compact manifold  $X$.   We introduced the Kodaira Spencer DG-Lie algebra  of the pair $KS_{(X,D)}$ in Example \ref{example DGLA pair}. Then, considering  ${\Theta}_X(-\log D)$ instead of   $\Theta_X$, we can proceed as in the case of  $KS_X$ of Corollary \ref{cor def X non ostruite}.

\begin{corollary}
Let $D$ be  a smooth divisor in a compact manifold  $X$, such that the logarithmic canonical bundle ${\Omega}^n_X(\log D)$ is trivial and the logarithmic  Hodge-de Rham spectral sequence degenerates at $E_1$-level. Then, $KS_{(X,D)}$ is homotopy abelian.
\end{corollary}

\begin{proof} 
\cite[Theorem 5.1 and Corollary 5.4]{donacoppie}. 
The proof is analogous to the one of Corollary \ref{cor def X non ostruite}.
Here, the Cartan homotopy is given by the   contraction of  logarithmic tangent vector and  logarithmic differentials:
\[\bi\colon KS_{(X,D)} \to M=\Hom_\C^{*}(\sA^{*,*}_X(\log D),\sA^{*,*}_X (\log D)).\]

Then, for any $ \eta \in KS_{(X,D)}$, $\bl_\eta \in H=\{  \varphi \in M \ \mid \varphi( \sA_X^{n,*}(\log D)  )\subset \sA_X^{n,*}(\log D)\} $, where $n$ is the dimension of $X$.

Next, the degeneration property implies that the inclusion $H \hookrightarrow M$  is injective in cohomology.
Since the logarithmic canonical bundle ${\Omega}^n_X(\log D)$ is trivial, 
 the cup product with a non trivial section of it, gives the isomorphisms $H^i(\Theta_X (\log D))\cong H^i(\Omega^{n-1}_X (\log D))$.
 
This implies that $H^*(KS_{(X,D)}) \to \Hom^{*}(H^0(\Omega^{n}_X(\log D)), H^*(\Omega^{n-1}_X(\log D)) )$ is injective  and so $KS_{(X,D)} \to M/H[-1]$ is also injective in cohomology.
Therefore, the hypothesis of   Theorem \ref{theorem abstract btt} are satisfied.

\end{proof}
 
\begin{remark}
 The degeneration hypothesis is satisfied for any globally  normal crossing divisor $D$ in a compact  K\"{a}hler  manifold  \cite[Theorem 8.35]{voisin}.
\end{remark}
 
 \begin{remark}
It can be also proved that $KS_{(X,D)}$ is homotopy abelian for a smooth divisor $D$ in a Calabi-Yau manifold $X$ \cite[Theorem 4.7]{donacoppie}.
In this case, the relevant spectral sequence  is the one associated with the Hodge filtration
\[
E_{1}^{p,q}=H^q(X,\Omega^{p}_X(\log D) \otimes \Oh_X(-D))  \Longrightarrow 
\mathbb{H}^{p+q}( X,\Omega^{\ast}_X(\log D)  \otimes \Oh_X(-D))
\]
and it degenerates at the $E_1$-level \cite[Section 4.3]{fujino1}.
 \end{remark}

 \begin{remark}
It is well known that the  Kodaira Spencer DG-Lie algebra $KS_{(X,D)}$ controls the  infinitesimal deformations of  the $(X,D)$. Then,  the infinitesimal deformations of the pair $(X,D)$ are unobstructed when $D$ is a smooth  divisor in a compact  K\"{a}hler  manifold $X$ such that ${\Omega}^n_X(\log D)$ is trivial  \cite[Corollary 4.5]{donacoppie} and when $D$ is  a smooth  divisor in a compact Calabi-Yau manifold $X$ \cite[Corollary 4.8]{donacoppie}.

  \end{remark}

\begin{remark} 

In general, if the ground field is an algebraically closed field
of characteristic 0,    we can replace the  DG-Lie algebra $KS_{(X,D)}$ with 
the DG-Lie algebra  $TW({\Theta}_X(-\log D)(\mathcal{U}))$ obtained  applying the Thom-Withney realisation to the semicosimplicial DG-Lie algebras
${\Theta}_X(-\log D)(\mathcal{U})$, for any affine open cover $\mathcal{U}$ of $X$ \cite[Theorem 4.3]{donacoppie}.
Also in this case,   $TW({\Theta}_X(-\log D)(\mathcal{U}))$   is homotopy abelian, for a smooth divisor $D$ in a smooth projective variety $X$ such that  ${\Omega}^n_X(\log D)$ is trivial  and when $D$ is  a smooth  divisor in a smooth projective  Calabi-Yau  variety.

\end{remark}

If $D$ is not smooth but only a simple normal crossing divisor, then $KS_{(X,D)}$ (or $TW({\Theta}_X(-\log D)(\mathcal{U}))$)    controls the locally trivial infinitesimal deformations of  the pair  $(X,D)$. Then, the computations above shows that the locally trivial infinitesimal deformations are unobstructed.

\subsection{Differential Batalin-Vilkovisky algebras}
The main references for this example are \cite{terilla,KKP} and \cite[Section 7]{donacoppie}.

\begin{definition}\label{def dbv} 
Let $k$ be a fixed odd integer.
A \emph{differential Batalin-Vilkovisky algebra} (dBV for short) of degree $k$ over $\K$ is the data
$(A, d, \Delta)$, where $(A,d)$ is a differential $\Z$-graded  commutative algebra with unit $1\in A$, 
and $\Delta$ is an operator of degree $-k$, such that $\Delta^2=0$, $\Delta (1)=0$ and 
\begin{multline*}
\Delta(abc)+\Delta(a)bc+(-1)^{\bar{a}\;\bar{b}} \Delta(b) a c+(-1)^{\bar{c}(\bar{a}+\bar{b})}
\Delta(c)ab=\\
=\Delta(ab)c +(-1)^{\bar{a} (\bar{b}+\bar{c})} \Delta(bc)a+(-1)^{\bar{b}\bar{c}}\Delta(ac)b.
\end{multline*}
\end{definition}
For any  graded dBV algebra
$(A, d, \Delta)$ of degree $k$, it is canonically defined a  DG-Lie algebra
$(L,d,[-,-])$, where: 
$L=A[k]$, $d_{L}=-d_A$ and,
\[
[a,b]= (-1)^{p}(\Delta(ab)-\Delta(a)b)-a\Delta(b),\qquad \forall \  a\in A^p.
\]
\smallskip

\begin{definition}
A dBV algebra $(A,d,\Delta)$  of degree $k$ has the \emph{degeneration property} if  for every $a_0 \in A$,
such that $d a_0=0$, there exists a sequence $a_i$, $i\geq 0$, with $\deg(a_i)=\deg(a_{i-1})-k-1$ and 
such that   
\[
 \Delta a_i= da_{i+1}, \qquad i\geq 0.
\] 
 \end{definition}

\begin{example}\label{examp 1 k complex}
 Let $(V,d,\Delta)$ be a $(1,k)$-bicomplex, where $k$ is an odd integer, i.e., $(V,d)$ is a DG vector space and  $\Delta\in \Hom^k_{\K}(V,V)$  such that  
\[  \Delta^2=0 \quad [d,\Delta]=d\Delta+\Delta d=0.\]
If the $d\Delta$-lemma holds, i.e.,
\[ 
\ker  d\cap \Delta(V)=\ker \Delta \cap d(V) = d \Delta (V),
\]
then $(V,d,\Delta)$ has the degeneration property. 
Indeed,  if $da_0=0$ we have $\Delta a_0\in d\Delta (V)$ and then there exists $b\in V$ such that 
$d\Delta(b)=\Delta a_0$. It is sufficient to take $a_1=\Delta(b)$ and $a_i=0$ for every $i\ge 2$. Note that the converse is not true in general. For instance, if $d=\Delta$, then  $(V,d,\Delta)$ has the degeneration property, while $\ker \Delta \cap d(V)=d \Delta (V)$ if and only if  $d= \Delta=0$. 
\end{example}

\begin{example} \label{example complex degeneration}
Let $(V,d,\Delta)$ be a $(1,k)$-bicomplex as in Example \ref{examp 1 k complex} and suppose the existence of an operator $f \in \Hom^{k-1}_{\K}(V,V)$ such that 
\[ \Delta=[d,f],\qquad [f,\Delta]=0.\]
Then $(V,d,\Delta)$ has the degeneration property.
Indeed, consider  a formal parameter $t$  and the associative graded algebra $\Hom^{*}_{\K}(V,V)[[t]]$. Then, we have
\[ e^{tf}de^{-tf}=e^{[tf,-]}d=d+t[f,d]=d-t\Delta\]
and therefore
\[ t\Delta e^{tf}=-e^{tf}d+de^{tf}=[d,e^{tf}].\]
Let $a\in V$ be such that $da=0$ and define the sequence $a_i$ by the rule
\[\sum_{i\ge 0}a_it^i=e^{tf}(a).\]
It implies $a_0=a$ and 
\[ \sum_{i\ge 0}t^{i+1}\Delta a_i=t\Delta e^{tf}a=de^{tf}a=
\sum_{i\ge 0}t^{i}da_i\]
and then $da_{i+1}=\Delta a_i$ for every $i$. 
\end{example}

 \begin{theorem}\label{theorem dbv degener implies homotopy abelian}
Let  $(A,d,\Delta)$ be  a dBV algebra with the degeneration property. Then, the associated DG-Lie algebra $L=A[k]
$ is homotopy abelian.
\end{theorem}

\begin{proof}  Here we only give a sketch of the proof. The original proof  can be found in \cite[Theorem 1]{terilla} or \cite[Theorem 4.14]{KKP} for $k=1$ and  in  \cite[Theorem 7.6]{donacoppie} for any odd $k$. 

Given a dBV algebra  $(A, d, \Delta)$   and 
 $t$ a formal central variable of (even) degree $1+k$, we can define the graded vector space $A[[t]]$  of formal power  series with coefficients in $A$ and the graded vector space  $A((t))=\bigcup_{p\in\Z} t^pA[[t]]$   of formal Laurent power series.
Note that    $d(t)=\Delta(t)=0$    and    $d-t\Delta$ is a  well-defined differential   on $A((t))$.

Let $F^\bullet$ be the filtration  on the complex $(A((t)), d-t\Delta)$  defined by
$F^p=t^pA[[t]]$, for every $p \in \Z$.  Note that $A((t))=\bigcup _{p\in \Z} F^p$ and $F^0=A[[t]]$ and 
the map $a\mapsto \dfrac{a}{t}$ induces  an isomorphism of DG-vector spaces $A\to F^{-1}/F^0$.

The degeneration property   implies that the inclusion $F^p \to A((t))$ is injective in cohomology, for every $p$, and, in particular,  
$F^0=A[[t]]\to A((t))$ is   injective in cohomology.

Consider $ M=\Hom^*_{\K}(A((t)),A((t)))$ and  $H=\{  \varphi \in M \ \mid \varphi( A[[t]])\subset A[[t]]\} $; then, the degeneration property implies that   $H \hookrightarrow M$ is  injective in cohomology (Hypotheis (3) of  Theorem \ref{theorem abstract btt}).

A tedious but straightforward computation shows that 
the map 
 \[\bi\colon L \to M=\Hom^*_{\K}(A((t)),A((t))),\qquad
a\longmapsto \bi_a(b)=\frac{1}{t}ab \]
is a Cartan homotopy \cite[Lemma 7.2]{donacoppie} and,  for any $a\in L$, $\bl_a \in H$, i.e.,  Hypotheis (1) and (2) of Theorem \ref{theorem abstract btt} are satisfied. Therefore, we only need to show that the morphism of DG-vector spaces $\bi: L \to M/H[-1]$  is  injective in cohomology  and this follows again by degeneration property.
Indeed, $M/H[-1]= \Hom^*_{\K}\left(A[[t]], \frac{A((t))}{A[[t]]}\right)[-1]$ and so we need  to prove that:
\[\bi\colon A\to \Hom^*_{\K}\left(A[[t]], \frac{A((t))}{A[[t]]}\right)[-k-1]\]
is injective in cohomology. 
It suffices to show  the injectivity for the composition  with the  evaluation at $1\in A[[t]]$, i.e.,    
 the injectivity in cohomology of the morphism 
\[
 A\to \frac{A((t))}{A[[t]]},\qquad a\mapsto \frac{a}{t},
\]
 or equivalently of the morphism
\[
\beta: \frac{F^{-1}}{F^0} \to \frac{A((t))}{F^0}. 
\]
Consider the diagram 
\[ \xymatrix{ 0 \ar[r] &  F^0\ar[r] \ar@{= }[d]& F^{-1}  \ar[d]_{ \alpha }\ar[r] &  \dfrac{F^{-1}} {F^0}  \ar[d]^\beta \ar[r] &0\\
0\ar[r]& F^0\ar[r] & A((t)) \ar[r ]&  \dfrac {A((t))} {F^0}   \ar[r] & 0; \\ }
\]
the degeneration property implies that $\alpha$ is injective in cohomology and so $\beta$ is also injective in cohomology.
 
\end{proof}

\begin{remark}

In \cite{bandieraNA}, the author uses similar techniques to extend this result to the context of commutative $BV_\infty$-algebras. In this case, the degeneration property implies that the associated $L_\infty[1]$-algebras is homotopy abelian \cite[Theorem 6.16]{bandieraNA}, see also \cite{BraunLazarev}.
\end{remark}

\begin{example}\cite[Section 5]{FiMaformal}  and \cite[Section 6]{Rugg}. 
Let $X$ be a complex manifold and $\Theta_X$ it tangent bundle. An holomorphic Poisson bi-vector on $X$ is an element $\pi \in H^0(X, \wedge^2 \Theta_X)$, such that $[\pi,\pi]_{SN}=0$, where $[\, , \ ]_{SN}$ denotes the Schouten-Nijenhuis bracket of polyvectorfields \cite{LMDT}. In this case, we say that the pair $(X, \pi)$ is a holomorphic Poisson manifold. 
Then, let $(A_X^*,d)$ be the de Rham complex, where $A_X^i= \bigoplus_{p+q=i} A^{p,q}_X $ and consider the map
\[\bi_{\pi} \in \Hom^{-2}(A_X^*,A_X^*)\qquad \bi_{\pi}(\alpha)=\pi\contr \alpha.\]   
Note that $[\bi_{\pi},\bi_{\pi}]=0$ and,  setting $\Delta=[d,\bi_{\pi}]$, we have 
\[ [\bi_{\pi},\Delta]=  [ \bi_{\pi},[d,\bi_{\pi}]]=\bi_{[\pi,\pi]_{SN}}=0.\]
According to Example \ref {example complex degeneration}, $(A_X^* ,d, \Delta)$ has the degeneration property and so Theorem \ref{theorem dbv degener implies homotopy abelian} implies that the associated DG-Lie algebra $(A_X^*[1],-d, [ , ]_\pi)$  is homotopy abelian (see \cite[Theorem 5.3]{FiMaformal} and \cite[Section 6]{Rugg}).
\end{example}

\subsection{The splitting property}
Let $(L, d , [\ , \ ])$ be a DG-Lie algebra. Denote by $L[1]$  the shifted graded vector space: $L[1]^i=L^{i+1}$ and by $L[1]^{\odot n}$ its $n$-th symmetric power.  Note that $1 \in L[1]^{\odot 0}=\K$. Then, we can consider the associated differential graded cocommutative coalgebra $(SL[1], \Delta, Q)$, where $\displaystyle SL[1]= \bigoplus_{n \geq 0} L[1]^{\odot n}$, $\Delta$ is the usual coproduct  and $Q$ the coderivation associated with $d$ and $ [\ , \ ]$.

More explicitly, let   $q_1(x)=-d(x)$  and $q_2(x \odot y)= -(-1)^i[x,y] $ for any $x \in L[1]^i$ and $y \in L[1]$, then 
$$
Q(v_1 \odot \ldots \odot  v_n)=\sum_{k=1}^2 \sum_{\sigma \in
S(k,n-k)} \epsilon(\sigma)q_k( v_{\sigma(1)} \odot \ldots \odot
v_{\sigma(k)})\odot v_{\sigma(k+1)} \odot \ldots \odot
v_{\sigma(n)},
$$
where  $S(p,q)$ denotes the set of unshuffles of type $(p,q)$ and 
$ \epsilon(\sigma)$ is the Koszul sign. It turns out that $Q^2=0$.

 Note that, any morphism $F:SL[1] \to SL[1]$ is uniquely determined by the corestriction 
 $pF=(f_0, f_1, f_2, \ldots): SL[1] \to SL[1]  \to L[1]$, where  $p: SL[1] \to L[1]^{\odot 1 }=L[1]$ is the natural projection and 
$f_i: L[1]^{\odot i } \to L[1]$ are the components of  $pF$.

Next, denote by $(\coder^*(SL[1]),[Q, \ ], [ \ , \ ])$ the DG-Lie algebra of coderivation of  $SL[1]$ and consider the surjective morphism of DG-vector spaces
\[
 \coder^*(SL[1]) \stackrel{b}{\to} L[1] , \qquad b(\alpha)= p\alpha (1), \quad \forall \alpha.
\]

\begin{corollary}[\cite{bandieraRendiconti}]\label{coroll L homotopy in coder}
In the above assumption, if $b\colon\coder^*(SL[1]) \to L[1]$ induces a surjective morphism in cohomology, then $L$ is homotopy abelian.

\end{corollary}

\begin{proof}
Here we only give a sketch of the proof as a consequence of the abstract BTT theorem, while the original proof contained in \cite{bandieraRendiconti} uses the theory of derived brackets.
Let $M= \coder^*(SL[1])$ and $H=\ker b \subset M$. The hypothesis implies that the inclusion $H\hookrightarrow M$ is injective in cohomology, i.e.,
Hypothesis (3) of  Theorem \ref{theorem abstract btt} is satisfied.
The map 
 \[\bi\colon L \to \coder^*(SL[1]),\qquad
a\longmapsto \bi_a(w)= x \odot w \]
is a properly defined Cartan homotopy.  Moreover, for any $a\in L$, $\bl_a \in H$, and so Hypothesis (1) and  (2)  are fulfilled.
This also implies that it is well defined the morphism of  DG vector spaces $\bi\colon L \to (M/ H)[-1]$, that is injective in cohomology.
Therefore, by   Theorem \ref{theorem abstract btt},
$L$ is homotopy abelian.

\end{proof}

\begin{remark}

According to \cite{bandieraRendiconti}, we say that $L$ has the splitting property if it satisfies the hypothesis of the previous theorem. We refer to \cite[Proposition 2.2]{bandieraRendiconti} for equivalent conditions.
In this paper,  the author proves this result for $L_\infty[1]$-algebras, analysing  the spectral sequence computing the Chevalley-Eilenberg cohomology  and showing the equivalence between the degeneration of the spectral sequence at the first page and  the homotopy abelianity   property.
 In \cite{ManettiFormal}, the author also analyses  the spectral sequence computing the Chevalley-Eilenberg cohomology of a DG-Lie algebra. In particular, he proved the equivalence between degeneration of the spectral sequence at the second page and formality property.
\end{remark}

\begin{remark} 
It can be also proved that the splitting property implies that a DG-Lie algebra $L$ is homotopy abelian  if the adjoint morphism  
\[ ad\colon L\to \Der^*_{\K}(L,L),\qquad ad_x(y)=[x,y],\]
is trivial in cohomology: a detailed proof will appear in the forthcoming paper \cite{homotopy}.
\end{remark}

 \subsection{Derived brackets of Lie type} 

Here we follow \cite[Section 5.6]{LMDT}.
Let $(M, [,],d)$ be a DG-Lie algebra, such that there exist a DG-Lie subalgebra $L$ and a graded vector space $A$ satisfying the following conditions:
\begin{enumerate}
\item [(i)] $M= L\oplus A$ as graded vector space;
\item [(ii)] $[a,b]=0$ for every $a, b \in A$;
\item [(iii)] $ [da,b]\in A$, for every $a,b \in A$.
\end{enumerate}
Consider  the projection $p: M \to A$  and the operators:
\[
\delta: A^i \to A^{i+1}, \qquad \delta a=-pda,
\]
\[
\{ - \ , \ - \}:A^{i-1} \times A^{j-1} \rightarrow A^{i+j-1}, \qquad   \{ a ,\,b \}:=-(-1)^{i}[da,b].
\]
A straightforward computation shows that $(A[-1], \delta, \{ - \ , \ - \})$ is a DG-Lie algebra \cite[Proposition 5.6.6]{LMDT}.

\begin{corollary}
If $L \rightarrow M$ is injective in cohomology, then the DG-Lie algebra $A[-1]=(A[-1], \delta, \{ - \ , \ - \})$ is homotopy abelian.

 \end{corollary}
 \begin{proof}
The linear map:
\[
\bi : A[-1]\to M \qquad \bi_a=\bi (a) =a
\]
is a Cartan homotopy. Indeed, for every $a,b\in A[-1]$, we have  $[\bi_a,\bi_{b}] =[a,b]=0$, by Condition (ii) above. Moreover, 
for every $a\in A[-1]^i$ and $b \in A[-1]$, 
 we have $0= d([a,b])= [da,b] + (-1)^{i-1}[a,db]$ and so 
\[
 \bi_{\{a,b\}}=  \{ a, \,b \}=-(-1)^{i}[da,b]=[a,db]=[\bi_a,d\bi_b].
\]
Note that, for every $a \in A[-1]$, we have
\[ 
p(\bl_a)=p(d \bi_a+\bi_{\delta a})= p(da -pda)=0
\]
and so  for every  $a \in A[-1]$, we have  $\bl_a\in L$, i.e., Hypothesis (2) of Theorem  \ref{theorem abstract btt}  is satisfied.
The last two hypotheses of Theorem  \ref{theorem abstract btt} are satisfied by assumption, and so we conclude that 
 $A[-1]=(A[-1], \delta, \{ - \ , \ - \})$ is homotopy abelian.
 \end{proof}

\begin{example} \cite[Section 5.6]{LMDT}.
Let $V$ and $W$ be two graded vector spaces over $\K$ and $\pi \in \Hom^1_\K (W,V)$, such that $0=d_{\Hom} \pi=d_W\pi + \pi d_V$. In \cite[Section 5.6]{LMDT}, the author introduced the notion of the derived bracket of $\pi$ on $\Hom^*_\K (V,W)[-1]$
defined as:
\[
[ - \ , - \ ]_\pi: (\Hom^*_\K (V,W)[-1])^i \times (\Hom^*_\K (V,W)[-1])^j \to (\Hom^*_\K (V,W)[-1])^{i+j}
\]
 \[
[f,g]_\pi=f\pi g - (-1)^{ij} g\pi f.
 \]
On the graded vector space  $\Hom^*_\K (V,W)[-1]$, it is also defined a differential $\delta$ as:
\[
\delta(f)= -d_W f - (-1)^i f d_V,  \qquad \forall \ f \in (\Hom^*_\K (V,W)[-1])^i. \]
Then, $A[-1]=(\Hom^*_\K (V,W)[-1], [ - \ , - \ ]_\pi , \delta)$ is a DG-Lie algebra.
To view this example in the above setting, it is enough to consider
the DG-Lie algebra $M=\Hom^*_\K (V\oplus W, V\oplus W)$ with differential given by $[D, \ - ]$ where 
\[
D=   \begin{pmatrix}
 d_V&-\pi\\
 0&d_W\\
\end{pmatrix} : V\oplus W \rightarrow V\oplus W.
\]
It turns out that $\{ - \ , \ - \} = [ - \ , - \ ]_\pi$; indeed, for every $a\in A[-1]^i$ and  $b\in A[-1]^j$, we have
\[\{ a,b\}=-(-1)^{i}[Da,b]= -(-1)^{i}[[D,a],b].
\]
By definition, 
\[\begin{split}[[D,a],b] &= [Da - (-1)^{i-1} aD, b]\\
&= Dab - (-1)^{i-1} aD b -(-1)^{i(j-1)}(bDa - (-1)^{i-1} baD)\\
&=Dab - (-1)^{i-1} aD b -(-1)^{i(j-1)}bDa - (-1)^{ij} baD.
\end{split}\]
Viewing $a$ and $b$ as
\[
a=   \begin{pmatrix}
 0&0\\
 a&0\\
\end{pmatrix}, \ \ b=   \begin{pmatrix}
 0&0\\
 b&0\\
\end{pmatrix}  : V\oplus W \rightarrow V\oplus W,
\]
we have
\[
Da=  \begin{pmatrix}
 -\pi a&0\\
 d_W a&0\\
\end{pmatrix}; 
\]
thus $Dab=abD=0$, and 
\[
bDa=  \begin{pmatrix}
 0&0\\
 b&0\\
\end{pmatrix} \cdot  \begin{pmatrix}
 -\pi a&0\\
 d_w a&0\\
\end{pmatrix} =\begin{pmatrix}
 0&0\\
 -b\pi  a&0\\
\end{pmatrix},  \qquad   aDb=   \begin{pmatrix}
 0&0\\
 -a\pi  b&0\\
\end{pmatrix}.
\]
Therefore,
\[\{ a,b\}=-(-1)^{i}[Da,b]= - aD b +(-1)^{ij}bDa
 =  a\pi b - (-1)^{ij} b\pi a= [a,b]_\pi .
\]

Finally,  if the inclusion $L \rightarrow M$ is injective in cohomology, then the 
DG-Lie algebra $A[-1]=(\Hom^*_\K (V,W)[-1], [ - \ , - \ ]_\pi , \delta)$ is homotopy abelian.

\end{example}

\section{Further applications}\label{section further appli}
 
According to Remark \ref{remark dopo teorema diagramma semireg}, if 
$L$, $ M$ and $H\subseteq M$  are  DG-Lie algebras, such that $\chi: H \hookrightarrow M$  is injective in cohomology and  $\bl_a\in H$ for every $a\in L$, we have a diagram:
\[\xymatrix{ TW(\alpha) \ar[d]^p \ar[r]^\phi  & TW(\chi)     \\
 L,}\]
where $TW(\chi) $ is an  homotopy abelian DG-Lie algebras and the vertical arrow is a quasi isomorphism.
In particular, we have a morphism:
\[
s: H^2(L) \xrightarrow{H^2(p)^{-1}} H^2(TW(\alpha))\xrightarrow{H^2(\phi)} H^2(TW(\chi))
\]
Since  $\Def_{TW(\chi)}$ is smooth, $s$ annihilates all the obstructions of $\Def_L$. 
 This idea has been applied in various deformation cases. For instance, 
in \cite[Proposition 4.6]{Periods}, the authors consider the deformations of compact K\"ahler manifolds and they prove the Kodaira's principle, ambient cohomology annihilates obstructions.
A semiregularity map annihilating all the obstructions to the infinitesimal deformations  of holomorphic maps of compact K\"ahler manifolds with fixed codomain 
were analysed in \cite[Corollary 4.14]{IaconoSemireg}.
In \cite[Theorem 11.1]{semireg}, it is proved that the Bloch's semiregularity map annihilates all the obstructions to embedded deformations of a local complete intersection, extending the proof given in \cite[Theorem 9.1]{ManettiSemireg}  for the embedded deformations of a smooth submanifold.

\begin{acknowledgement}
The author wishes
to thank Prof. Marco Manetti for useful discussion during the preparation of the paper and for the draft uncompleted version of \cite{LMDT}. The paper is largely inspired by the lectures held by Prof. Marco Manetti at GAP-XV at State College.
The author was partially supported by Fondi di Ateneo dell'Universit\`a degli Studi di Bari: \lq\lq Spazi di Moduli e Algebre di Lie Differenziali Graduate\rq\rq. The author thanks the referee for improvements in the presentation of the paper and Prof. O. Fujino for pointing my attention to \cite{fujino1}.
\end{acknowledgement}

\end{document}